\documentclass[12pt]{amsart}
\usepackage{amssymb,latexsym,amsmath,amscd,amsthm,graphicx, color}
\usepackage[all]{xy}
\usepackage{pgf,tikz}
\usepackage{mathrsfs}
\usetikzlibrary{arrows}
\usepackage[left=0.6 in, top=0.6 in, right=0.6 in, bottom=0.3 in]{geometry}
\definecolor{uuuuuu}{rgb}{0.26666666666666666,0.26666666666666666,0.26666666666666666}
\definecolor{xdxdff}{rgb}{0.49019607843137253,0.49019607843137253,1.}
\definecolor{ffqqqq}{rgb}{1.,0.,0.}
\definecolor{ffqqqq}{rgb}{1.,0.,0.}
\definecolor{ffxfqq}{rgb}{1.,0.4980392156862745,0.}

\definecolor{uuuuuu}{rgb}{0.26666666666666666,0.26666666666666666,0.26666666666666666}
\definecolor{qqwuqq}{rgb}{0.,0.39215686274509803,0.}
\definecolor{zzttqq}{rgb}{0.6,0.2,0.}
\definecolor{xdxdff}{rgb}{0.49019607843137253,0.49019607843137253,1.}
\definecolor{qqqqff}{rgb}{0.,0.,1.}
\definecolor{cqcqcq}{rgb}{0.7529411764705882,0.7529411764705882,0.7529411764705882}
\definecolor{sqsqsq}{rgb}{0.12549019607843137,0.12549019607843137,0.12549019607843137}

\theoremstyle{plain}

\newtheorem{theorem}[subsection]{Theorem}

\newtheorem{defi}[subsection]{Definition}
\newtheorem{prop}[subsection]{Proposition}
\newtheorem{prop1}[subsubsection]{Proposition}

\theoremstyle{definition}

\newtheorem{exam}[subsection]{Example}

\newtheorem{remark}[subsection]{Remark}
\newtheorem{remark1}[subsubsection]{Remark}

\newtheorem{conj}[subsection]{Conjecture}

\newcommand{\ci}{\subseteq}
\newcommand{\sci}{\subset}

\newcommand{\set}[1]{\{#1\}}

\newcommand{\ga}{\alpha}
\newcommand{\gb}{\beta}

\newcommand{\gd}{\delta}

\newcommand{\tbf}{\textbf}
\newcommand{\tit}{\textit}

\newcommand{\D}[1]{\mathbb{#1}}

\newcommand{\te}{\text}

\begin{document}

\title{Quantization for a set of discrete distributions on the set of natural numbers}

\address{School of Mathematical and Statistical Sciences\\
University of Texas Rio Grande Valley\\
1201 West University Drive\\
Edinburg, TX 78539-2999, USA.}

\email{\{$^1$juan.gomez15, $^2$haily.martinez01, $^3$mrinal.roychowdhury\}@utrgv.edu}
\email{\{$^4$alexis.salazar01, $^5$daniel.vallez01\}@utrgv.edu}

\author{$^1$Juan Gomez}
 \author{$^2$Haily Martinez}
\author{$^3$Mrinal K. Roychowdhury}
\author{$^4$Alexis Salazar}
\author{$^5$Daniel J. Vallez}

%
%
%
%
%
%

%


\date{}
\maketitle

\pagestyle{myheadings}\markboth{J. Gomez, H. Martinez, M. Roychowhury, A. Salazar, D. Vallez}{Quantization for a set of discrete distributions on the set of natural numbers}
\begin{abstract}
 The quantization scheme in probability theory deals with finding a best approximation of a given probability distribution by a probability distribution that is supported on finitely many points. In this paper, first we state and prove a theorem, and then give a conjecture. We verify the conjecture by a few examples. Assuming that the conjecture is true, for a set of discrete distributions on the set of natural numbers we have calculated the optimal sets of $n$-means and the $n$th quantization errors for all positive integers $n$. In addition, the quantization dimension is also calculated.
\end{abstract}

\section{Introduction}
The most common form of quantization is rounding-off. Its purpose is to reduce the cardinality of the representation space, in particular, when the input data is real-valued. It has broad applications in communications, information theory, signal processing and data compression (see \cite{GG, GL1, GL2, GN, P, Z1, Z2}).  Let $\D R^d$ denote the $d$-dimensional Euclidean space equipped with the Euclidean norm $\|\cdot\|$, and let $P$ be a Borel probability measure on $\D R^d$. Then, the $n$th \textit{quantization
error} for $P$, with respect to the squared Euclidean distance, is defined by
\begin{equation*} \label{eq1} V_n:=V_n(P)=\inf \Big\{V(P; \ga) : \ga \subset \mathbb R^d, 1\leq  \text{ card}(\ga) \leq n \Big\},\end{equation*}
where $V(P; \ga)= \int \mathop{\min}\limits_{a\in\alpha} \|x-a\|^2 dP(x)$ represents the distortion error due to the set $\ga$ with respect to the probability distribution $P$, and for a set $A$, $\te{card}(A)$ represents the cardinality of the set $A$.
A set $\ga$ for which the infimum occurs and contains no more than $n$ points is called an \tit{optimal set of $n$-means}, and is denoted by $\ga_n:=\ga_n(P)$. The elements of an optimal set are also called as \tit{optimal quantizers}.
It is known that for a Borel probability measure $P$ if its support contains infinitely many elements and $\int \| x\|^2 dP(x)$ is finite, then an optimal set of $n$-means always exists and has exactly $n$-elements \cite{AW, GKL, GL1, GL2}.
The number
\begin{equation} \label{eq0000} D (P):=\lim_{n\to \infty}  \frac{2\log n}{-\log V_{n}(P)},
\end{equation}
if it exists, is called the \tit{quantization dimension of $P$}. Quantization dimension measures the speed at which the specified measure of the error goes to zero as $n$ tends to infinity.
For a finite set $\ga \sci \D R^d$ and $a\in \ga$, by $M(a|\ga)$ we denote the set of all elements in $\D R^d$ which are the nearest to $a$ among all the elements in $\ga$, i.e.,
$M(a|\ga)=\set{x \in \D R^d : \|x-a\|=\min_{b \in \ga}\|x-b\|}.$
$M(a|\ga)$ is called the \tit{Voronoi region} generated by $a\in \ga$.
On the other hand, the set $\set{M(a|\ga) : a \in \ga}$ is called the \tit{Voronoi diagram} or \tit{Voronoi tessellation} of $\D R^d$ with respect to the set $\ga$.
The following proposition provides further information on the Voronoi regions generated by an optimal set of $n$-means (see \cite{GG, GL2}).
\begin{prop} \label{prop10}
Let $\alpha$ be an optimal set of $n$-means, $a \in \alpha$, and $M (a|\ga)$ be the Voronoi region generated by $a\in \ga$, i.e.,
\[M(a|\ga)=\{x \in \mathbb R^d : \|x-a\|=\min_{b \in \alpha} \|x-b\|\}.\]
Then, for every $a \in\alpha$,
\begin{itemize}
\item[(i)] $P(M(a|\ga))>0$,
\item[(ii)] $ P(\partial M(a|\ga))=0$,
\item[(iii)] $a=E(X : X \in M(a|\ga))$, and
\item[(iv)] $P$-almost surely the set $\set{M(a|\ga) : a \in \ga}$ forms a Voronoi partition of $\D R^d$.
\end{itemize}
\end{prop}
From the above proposition, we can say that if $\ga$ is an optimal set of $n$-means for $P$, then each $a\in \ga$ is the conditional expectation of the random variable $X$ given that $X$ takes values on the Voronoi region of $a$. Sometimes, we also refer to such an $a\in \ga$ as the centroid of its own Voronoi region.
In this regard, interested readers can see \cite{DFG, DR, R1}.

A vector $(p_1, p_2, p_3, \cdots)$ is called a probability distribution if $0<p_j<1$ for all $j\in \D N$ and $j\geq 0$ such that $\sum_{j\geq 1}p_j=1$. Notice that $(p_1, p_2, p_3, \cdots)$ can be a finite, or an infinite vector, where by a finite vector it is meant that the number of coordinates in the vector is a finite number, otherwise it is called an infinite vector.

For a Borel probability measure $P$ on the set $\D R$ of real numbers let $U$ be the largest open subset of $\D R$ such that $P(U)=0$, then $\D R\setminus U$ is called the \tit{support} of the probability measure $P$. For example, when an unbiased die is thrown one time, then $P$ is a Borel probability measure on the real line with
\[\te{support}(P)=\set{1, 2, 3, 4, 5, 6},\]
and the associated probability distribution $(\frac 16, \frac 16, \frac 16, \frac 16, \frac 16, \frac 16)$.
When an unbiased coin is tossed two times, then $P$ is a Borel probability measure on $\D R^2$ with
\[\te{support}(P)=\set{(1,1), (1, 2), (2,1), (2,2)},\]
and the associated probability distribution $(\frac 14, \frac 14, \frac 14, \frac 14)$,
where $1$ stands for `Head' and $2$ stands for `Tail'.

\begin{defi} \label{defi1}
Let $(p_1, p_2, \cdots, p_{k-1})$ be a permutation of the set $\set{\frac 12, \frac 1{2^2}, \cdots, \frac 1{2^{k-1}}}$, where $k\in \D N:=\set{1, 2, \cdots}$ with $k\geq 2$. Define a probability measure $P$ on the set $\D R$ of real numbers with the support the set $\D N$ of natural numbers as follows:
\[P:=\sum_{j=1}^{k-1} p_j\gd_j+\sum_{j=k}^\infty \frac 1{2^j} \gd_j,\]
where for $x\in \D R$ the function $\gd_x$ represents the dirac measure, i.e., for any subset $A\ci \D R$, we have $\gd_x(A)=1$ if $x\in A$, and zero otherwise.
\end{defi}

Let us now state the following theorem and the conjecture.
\begin{theorem}\label{main}
Let $P:=\sum_{j=1}^{k-1} p_j\gd_j+\sum_{j=k}^\infty\frac 1{2^j}\gd_j$ be the probability measure as defined by Definition~\ref{defi1}. Let $\set{a_1, a_2, \cdots, a_{n-3}, a_{n-2}, a_{n-1}, a_{n}}$ be an optimal set of $n$-means with $n\geq k+2$. Suppose that $a_1=1, a_2=2, \cdots, a_{n-3}=n-3$. Then, either $a_{n-2}=n-2, a_{n-1}=Av[n-1, n], a_{n}=Av[n+1, \infty)$, or $a_{n-2}=Av[n-2, n-1], a_{n-1}=Av[n, n+1], a_{n}=Av[n+2, \infty)$ with quantization error $V_{n}=\frac{2^{3-n}}{3}$,
where for any $k, \ell\in \D N$, $Av[k, \ell]$ and $Av[k, \infty)$ are defined in the next section.
\end{theorem}

\begin{exam}
Let $(\frac 1{2^3}, \frac 1{2^2}, \frac 12)$ be a permutation of the set $\set{\frac 12, \frac 1{2^2}, \frac 1{2^3}}$. Write
\[P:=\frac 1{2^3}\gd_1+\frac 1{2^2} \gd_2+\frac 1{2} \gd_3+\sum_{j=4}^\infty \frac 1{2^j} \gd_j.\]
Then, $P$ is a Borel probability measure on $\D R$ with support the set $\D N$ of natural numbers. Let us assume that $\set{a_1, a_2, \cdots, a_n}$ is an optimal set of $n$-means for $n=6$. If $a_1=1$, $a_2=2$, and $a_3=3$, then by Theorem~\ref{main}, we must have the set $\set{a_4, a_5, a_6}$ equals either the set $\set{4, Av[5,6], Av[7, \infty)}$, or the set $\set{Av[4,5], Av[6,7], Av[8, \infty)}$ with quantization error $V_6=\frac 1{24}$.
\end{exam}

\begin{conj} \label{conj1}
Let $P:=\sum_{j=1}^{k-1} p_j\gd_j+\sum_{j=k}^\infty \frac 1{2^j}\gd_j$ be the probability measure as defined by Definition~\ref{defi1}. Let $\set{a_1, a_2, a_3, \cdots, a_n}$ be an optimal set of $n$-means with $n\geq k+2$ such that $a_1<a_2<\cdots<a_n$. Then, $a_1=1, a_2=2, \cdots, a_{n-3}=n-3$.
\end{conj}

In this paper, first we give a complete proof of Theorem~\ref{main}. Then, we verify the conjecture by two discrete distributions as mentioned in Remark~\ref{rem1} and Remark~\ref{rem2}. Under the assumption that the conjecture is true, we calculate the optimal sets of $n$-means and the $n$th quantization errors for the two discrete distributions for all $n\in \D N$.
Once the quantization error is known, the quantization dimension can easily be calculated, see Proposition~\ref{prop3333}. In addition, in the last section, we give a proposition Proposition~\ref{prop51}. By this proposition, we deduce that if  $(p_1, p_2, \cdots, p_{k-1})$ is not a permutation of the set $\set{\frac 12, \frac 1{2^2}, \cdots, \frac 1{2^{k-1}}}$, where $k\in \D N:=\set{1, 2, \cdots}$ with $k\geq 2$, then the conjecture is not true. The general proof of the conjecture is not known yet. Such a problem still remains open.

\section{Preliminaries}
Let $\D N:=\set{1, 2, 3, \cdots}$ be the set of natural numbers. Let $(p_1, p_2, p_3,\cdots)$, where $0<p_j<1$ for all $j\in \D N$ and $\sum_{j=1}^\infty p_j=1$, be a probability distribution. Let  \[P=\sum_{j=1}^\infty p_j \gd_j, \]
where $\gd_j$ is the dirac measure as given in Definition~\ref{defi1}. Then, $P$ is a discrete probability measure on the set $\D R$ of real numbers with the support the set of natural numbers $\D N$ associated with the probability distribution $(p_1, p_2, p_3,\cdots)$.
In fact, if $X$ is a random variable associated with the probability measure $P$, and $f$ is the probability mass function, then we have
 \[P(X=j)=f(j)=p_j.\]
Define the following notations:
For $k, \ell\in \D N$, where $k\leq \ell$, write
\[[k, \ell]:=\set{n : n \in \D N \te{ and } k\leq n\leq \ell}, \te{ and } [k, \infty):=\set{n : n\in \D N \te{ and } n\geq k}.\]
Further, write
\[Av[k, \ell]: =E\Big (X : X \in [k, \ell]\Big)=\frac{\sum _{n=k}^{\ell} p_n n}{\sum_{n=k}^{\ell}p_n}, \  Av[k, \infty): =E\Big (X : X \in [k, \infty)\Big)=\frac{\sum _{n=k}^{\infty} p_n n}{\sum_{n=k}^{\infty}p_n},\]
\[Er[k, \ell]:=\sum _{n=k}^{\ell} p_n \Big(n-Av[k, \ell]\Big)^2, \te{ and } Er[k, \infty):=\sum _{n=k}^{\infty} p_n \Big(n-Av[k,\infty)\Big)^2.\]
 Notice that
$E(X):=E(X : X \in \te{supp}(P)) =\sum _{n=1}^\infty p_n n$, and so the optimal set of one-mean is the set $\set{\sum _{n=1}^\infty p_n n}$ with quantization error
\[V(P)=\sum _{n=1}^{\infty} p_n (n-E(X))^2.\]
In the following sections, we give the main results of the paper.

\section{Proof of Theorem~\ref{main} and verifications of Conjecture~\ref{conj1}} \label{sec2}
In this section, in the following subsections first we prove Theorem~\ref{main}, and then by two different examples, we verify that Conjecture~\ref{conj1} is true. Then, we state and prove Proposition~\ref{prop3333}, which gives the quantization dimension of the probability measure $P$.

\subsection{\tbf{Proof of Theorem~\ref{main}}}
Let $n\geq k+2$, where $k\geq 2$.
The distortion error due to the set $\gb:=\set{1, 2, \cdots, n-3, n-2, Av[n-1, n], Av[n+1, \infty)}$ is given by
\[V(P; \gb)=Er[n-1, n]+Er[n+1, \infty)=\frac{2^{3-n}}{3}.\]
Since $V_{n}$ is the quantization error for $n$-means, we have $V_n\leq \frac{2^{3-n}}{3}$.
Let $\ga:=\set{a_1, a_2, a_3, \cdots, a_{n}}$ be an optimal set of $n$-means, where $1\leq a_1<a_2<a_3<\cdots<a_n<\infty$. Assume that $a_1=1, a_2=2, \cdots, a_{n-3}=n-3$. Then, the Voronoi region of $a_{n-2}$ must contain the element $n-2$. Suppose that the Voronoi region of $a_{n-2}$ contains the set $\set{n-2, n-1, n}$.
Then,
\[V_{n}\geq Er[n-2, n]=\frac{13}{7}  2^{1-n}>\frac{2^{3-n}}{3}\geq V_{n},\]
which is a contradiction. Hence, we can assume that the Voronoi region of $a_{n-2}$ contains only the set $\set{n-2}$ or the set $\set{n-2, n-1}$.
Let us consider the following two cases:

\tit{Case~1. The Voronoi region of $a_{n-2}$ contains only the set $\set{n-2}$.}

 Then, the Voronoi region of $a_{n-1}$ must contain the element $n-1$. Suppose that the Voronoi region of $a_{n-1}$ contains the set $\set{n-1, n, n+1, n+2}$. Then,
\[V_n\geq \frac{97}{15}\ 2^{-n-1}>V_n,\]
which is a contradiction. Assume that the Voronoi region of $a_{n-1}$ contains only the set $\set{n-1, n, n+1}$. Then, as the Voronoi region of $a_{n}$ contains the set $\set{k : k\geq n+2}$,
\[V_n\geq Er[n-1, n+1]+Er[n+2, \infty)=\frac 57 2^{2-n}>V_n,\]
which leads to a contradiction. Next, assume that the Voronoi region of $a_{n-1}$ contains only the set $\set{n-1}$. Then, the Voronoi region of $a_{n}$ contains the set $\set{k : k \geq n}$ yielding
\[V_{n}=Er[n, \infty)=2^{2-n}>V_n,\]
which gives a contradiction. This yields the fact that the Voronoi region of $a_{n-1}$ contains only the set $\set{n-1, n}$, and hence, the Voronoi region of $a_{n}$ contains only the set $\set{k : k\geq n+1}$. Thus, in this case we have
$a_{n-2}=n-2, a_{n-1}=Av[n-1, n], a_{n}=Av[n+1, \infty)$ with quantization error $V_n=\frac{2^{3-n}}{3}.$

\tit{Case~2. The Voronoi region of $a_{n-2}$ contains only the set $\set{n-2, n-1}$.}

 Then, the Voronoi region of $a_{n-1}$ must contain the element $n$. Suppose that the Voronoi region of $a_{n-1}$ contains the set $\set{n, n+1, n+2, n+3}$. Then,
\[V_n\geq Er[n-2, n-1]+Er[n, n+3]=\frac{59} 5 2^{-n-2}>\frac{2^{3-n}}{3}\geq V_{n},\]
which leads to a contradiction. Assume that the Voronoi region of $a_{n-1}$ contains only the set $\set{n, n+1, n+2}$. Then, as the Voronoi region of $a_{n}$ contains the set $\set{k : k\geq n+3}$,
\[V_{n}\geq Er[n-2, n-1]+Er[n, n+2]+Er[n+3, \infty)= \frac{29}{21}2^{1-n}>V_{n},\]
which leads to a contradiction. Hence, we can assume that the Voronoi region of $a_{n-1}$ contains only the set $\set{n}$, or the set $\set{n, n+1}$. If the Voronoi region of $a_{n-1}$ contains only the set $\set{n}$, then
\[V_{n}\geq Er[n-2, n-1]+Er[n+1, \infty)=\frac{5} 3 2^{1-n}>V_n,\]
which leads to a contradiction. Hence, we can assume that the Voronoi region of $a_{n-1}$ contains only the set $\set{n, n+1}$, and the Voronoi region of $a_{n}$ contains only the set $\set{k : k\geq n+2}$ yielding
$a_{n-2}=Av[n-2, n-1], a_{n-1}=Av[n, n+1], a_{n}=Av[n+2, \infty)$ with quantization error $V_{n}=\frac{2^{3-n}}{3}$.

Case~1 and Case~2 together give the optimal sets of $n$-means and the $n$th quantization errors for all positive integers $n$. Thus, the proof of the theorem Theorem~\ref{main} is completed. \qed

\subsection{Verification of Conjecture~\ref{conj1} when $(p_1, p_2, p_3, p_4, \cdots)=(\frac 1 {2^2}, \frac 1 2, \frac 1{2^3}, \frac 1{2^4}, \frac 1{2^5}, \cdots)$} \label{subsection111}
In this case the probability mass function $f$ for the probability measure $P$ on the set of real numbers $\D R$ is given by
\[f(j)=\left\{\begin{array} {ll}
\frac 1{2^2} & \te { if } j=1,\\
\frac 12 & \te { if } j=2,\\
\frac 1{2^n} & \te{ if } j =n \te { for } n\in \D N \te{ and } n\neq 1, 2,\\
0 &  \te{ otherwise}.
\end{array}
\right.\]
Notice that here $k=3$ and $(p_1,\cdots, p_{k-1})=(p_1, p_2)=(\frac 1{2^2}, \frac 12)$, where $k\in \D N$ as defined by Definition~\ref{defi1}.

Let us now prove the following proposition.

\begin{prop1} \label{prop2.1}
Let $n\geq 5$, and let $\ga_n$ be an optimal set of $n$-means for the probability measure $P$ given by
\[P=\frac 1{2^2}\gd_1+\frac 12 \gd_2+\sum_{j=3}^\infty \frac 1{2^j} \gd_j.\] Then, $\ga_n$ must contain the set $\set{1, 2, \cdots, (n-3)}$.
\end{prop1}
\begin{proof}
The distortion error due to the set $\gb:=\set{1, 2, \cdots, (n-2), Av[n-1, n], Av[n+1, \infty)}$ is given by
\[V(P; \gb)=Er[n-1, n]+Er[n+1, \infty)=\frac{2^{3-n}}{3}.\]
Since $V_n$ is the quantization error for $n$-means, we have $V_n\leq \frac{2^{3-n}}{3}$.
Let $\ga_n:=\set{a_1, a_2, \cdots, a_n}$ be an optimal set of $n$-means such that $1\leq a_1<a_2<\cdots<a_n<\infty$. We show that $a_1=1, a_2=2, \cdots, a_{n-3}=n-3$. We prove it by induction.
Notice that the Voronoi region of $a_1$ must contain the element $1$. Suppose that the Voronoi region of $a_1$ also contains the element $2$. Then,
\[V_n>\sum_{j=1}^2f(j)(j-Av[1,2])^2=\frac{1}{6}\geq \frac{2^{3-n}}{3} \geq V_n,\]
which is a contradiction. Hence, we can conclude that the Voronoi region of $a_1$ contains only the element $1$ yielding $a_1=1$.
Thus, we can deduce that there exists a positive integer $\ell$, where $1\leq \ell<n-3$, such that $a_1=1, a_2=2, \cdots, a_\ell=\ell$.
We now show that $a_{\ell+1}=\ell+1$. Notice that the Voronoi region of $a_{\ell+1}$ must contain $\ell+1$. Suppose that the Voronoi region of $a_{\ell+1}$ also contains the element $\ell+2$. Then, we have
\[V_n>\sum_{j=\ell+1}^{\ell+2}\frac 1 {2^j}(j-Av[\ell+1, \ell+2])^2=Er[\ell+1, \ell+2]=\frac{2^{-\ell-1}}{3}\geq \frac{2^{3-n}}{3}\geq V_n,\]
which is a contradiction. Hence, we can conclude that the Voronoi region of $a_{\ell+1}$ contains only the element $\ell+1$ yielding $a_{\ell+1}=\ell+1$. Notice that $2\leq \ell+1 \leq n-3$. Thus, by the Principle of Mathematical Induction, we deduce that  $a_1=1, a_2=2, \cdots, a_{n-3}=n-3$. Thus, the proof of the proposition is complete.
\end{proof}

\begin{remark1} \label{rem1}
Proposition~\ref{prop2.1} verifies that the conjecture Conjecture~\ref{conj1} is true.
\end{remark1}

\subsection{Verification of Conjecture~\ref{conj1} when $(p_1, p_2, p_3, p_4, \cdots)=(\frac 1 {2^3}, \frac 1 {2^2},  \frac 1{2}, \frac 1{2^4}, \frac 1{2^5}, \cdots)$} \label{subsection112}
In this case the probability mass function $f$ for the probability measure $P$ on the set of real numbers $\D R$ is given by
\[f(j)=\left\{\begin{array} {ll}
\frac 1{2^3} & \te { if } j=1,\\
\frac 12 & \te { if } j=3,\\
\frac 1{2^n} & \te{ if } j =n \te { for } n\in \D N \te{ and } n\neq 1, 3,\\
0 &  \te{ otherwise}.
\end{array}
\right.\]
Notice that here $k=4$ and $(p_1,\cdots, p_{k-1})=(p_1, p_2, p_3)=(\frac 1{2^3}, \frac 1{2^2}, \frac 12)$, where $k\in \D N$ as defined by Definition~\ref{defi1}.

Let us now prove the following proposition.

\begin{prop1} \label{prop2.2}
Let $n\geq 6$, and let $\ga_n$ be an optimal set of $n$-means for the probability measure $P$ given by
\[P=\frac 1{2^3}\gd_1+\frac 1{2^2} \gd_2+\frac 1{2} \gd_3+\sum_{j=4}^\infty \frac 1{2^j} \gd_j.\] Then, $\ga_n$ must contain the set $\set{1, 2, \cdots, (n-3)}$.
\end{prop1}
\begin{proof}
The distortion error due to the set $\gb:=\set{1, 2, \cdots, (n-2), Av[n-1, n], Av[n+1, \infty)}$ is given by
\[V(P; \gb)=Er[n-1, n]+Er[n+1, \infty)=\frac{2^{3-n}}{3}.\]
Since $V_n$ is the quantization error for $n$-means, we have $V_n\leq \frac{2^{3-n}}{3}$.
Let $\ga_n:=\set{a_1, a_2, \cdots, a_n}$ be an optimal set of $n$-means such that $1\leq a_1<a_2<\cdots<a_n<\infty$. We show that $a_1=1, a_2=2, \cdots, a_{n-3}=n-3$. We prove it by induction.
The Voronoi region of $a_1$ must contain the element $1$. Suppose that the Voronoi region of $a_1$ also contains the element $2$. Notice that the remaining elements of the set of natural numbers are contained in the union of the Voronoi regions of $a_2, a_3, \cdots, a_n$ with positive distortion error yielding
\[V_n>\sum_{j=1}^2f(j)(j-Av[1,2])^2=\frac{1}{12}\geq \frac{2^{3-n}}{3} \geq V_n,\]
which is a contradiction. Hence, we can conclude that the Voronoi region of $a_1$ contains only the element $1$, yielding $a_1=1$.
Thus, we can deduce that there exists a positive integer $\ell$, where $1\leq \ell<n-3$, such that $a_1=1, a_2=2, \cdots, a_\ell=\ell$.
We now show that $a_{\ell+1}=\ell+1$. Notice that the Voronoi region of $a_{\ell+1}$ must contain $\ell+1$. Suppose that the Voronoi region of $a_{\ell+1}$ also contains the element $\ell+2$. Then, proceeding in the similar lines as given in Proposition~\ref{prop2.1}, we can see that a contradiction arises.
Hence, we can conclude that the Voronoi region of $a_{\ell+1}$ contains only the element $\ell+1$ yielding $a_{\ell+1}=\ell+1$. Notice that $2\leq \ell+2\leq n-3$. Thus, by the Principle of Mathematical Induction, we deduce that  $a_1=1, a_2=2, \cdots, a_{n-3}=n-3$. Thus, the proof of the proposition is complete.
\end{proof}

\begin{remark1}\label{rem2}
Proposition~\ref{prop2.2} verifies that the conjecture Conjecture~\ref{conj1} is true.
\end{remark1}

\begin{prop1}\label{prop3333}
Let $P:=\sum_{j=1}^{k-1} p_j\gd_j+\sum_{j=k}^\infty \frac 1{2^j}\gd_j$ be the probability measure as defined by Definition~\ref{defi1}. Assume that Conjecture~\ref{conj1} is true. Then, the quantization dimension $D(P)$ exists and equals zero.
\end{prop1}
\begin{proof}
By Theorem~\ref{main} and under the assumption that Conjecture~\ref{conj1} is true, the $n$th quantization error for any positive integer $n\geq k+2$ for the probability measure $P$, defined by Definition~\ref{defi1}, is obtained as $V_n(P)=\frac{2^{3-n}}{3}$. Hence, using the formula \eqref{eq0000}, we have
$D(P)=0.$
\end{proof}

\section{Optimal quantization for the two probability distributions described in Section~\ref{sec2}}
In this section, in the following two subsections we determine the optimal sets of $n$-means and the $n$th quantization errors for all positive integers $n\geq 2$ for the two probability measures $P$ given in Subsection~\ref{subsection111} and Subsection~\ref{subsection112} under the assumption that Conjecture~\ref{conj1} is true.

\subsection{Optimal quantization for $P$ when $(p_1, p_2, p_3, p_4, \cdots)=(\frac 1 {2^2}, \frac 1 2, \frac 1{2^3}, \frac 1{2^4}, \frac 1{2^5}, \cdots)$} \label{subsection2}
Let us give the results in the following propositions.

\begin{prop1} \label{prop43}
The optimal set of two-means is given by $\set{Av[1,3], Av[4, \infty)}$ with quantization error $V_2=\frac{17}{28}$.
\end{prop1}
\begin{proof} We see that $Av[1,3]=\frac{13}{7}$, and $Av[4, \infty)=5$.
Since $3<\frac 12(\frac{13}{7}+5)=\frac{24}{7}<4$, the distortion error due to the set $\gb:=\set{\frac{13}{7}, 5}$ is given by
\[V(P; \gb)=Er[1,3]+Er[4, \infty)=\frac{17}{28}.\]
Since $V_2$ is the quantization error for two-means, we have $V_2\leq\frac{17}{28}$. Let $\ga:=\set{a_1, a_2}$ be an optimal set of two-means such that $a_1<a_2$. Since the points in an optimal set are the conditional expectations in their own Voronoi regions, we have $1\leq a_1<a_2<\infty$. Notice that the Voronoi region of $a_1$ must contain $1$. Suppose that the Voronoi region of $a_1$ contains the set $\set{1, 2, 3, 4}$. Then,
\[V_2\geq \sum_{j=1}^4f(j) (j-Av[1,4])^2=Er[1,4]=\frac{5}{8}>V_2,\]
which yields a contradiction. Hence, we can assume that the Voronoi region of $a_1$ contains only the set $\set{1}$ or $\set{1,2}$, or the set $\set{1,2,3}$.
Suppose that the Voronoi region of $a_1$ contains only the set $\set{1}$, and so the Voronoi region of $a_2$ contains the set $\set{n : n\geq 2}$. Then, we have
\[V_2=Er[2, \infty)=\frac{7}{6}>V_2,\]
which leads to a contradiction. Hence, we can assume that the Voronoi region of $a_1$ contains only the set $\set{1, 2}$, or the set $\set{1,2,3}$. Suppose that the Voronoi
region of $a_1$ contains only the set $\set{1,2}$. Then, the Voronoi region of $a_2$ contains $\set{3, 4, 5, \cdots}$ yielding
\[V_2=Er[1,2]+Er[3, \infty)=\frac{2}{3}>V_2,\]
which gives a contradiction. Hence, we can conclude that the Voronoi region of $a_1$ contains only the set $\set{1, 2,3}$, and the Voronoi region of $a_2$ contains the set $\set{j : j\geq 4}$ yielding
\[a_1=Av[1, 3]\te{ and } a_2=Av[4, \infty) \te{ with quantization error } V_2=Er[1,3]+Er[4, \infty)=\frac{17}{28}.\]
Thus, the proof of the proposition is complete.
\end{proof}
\begin{prop1} \label{prop111}
The set $\set{Av[1,2], Av[3,4], Av[5, \infty)}$  forms the optimal set of three-means with quantization error $V_3=\frac{1}{3}$.
\end{prop1}
\begin{proof}
The distortion error due to set $\gb:=\set{Av[1,2], Av[3,4], Av[5, \infty)}$ is given by
\[V(P; \gb)=Er[1,2]+Er[3,4]+Er[5, \infty)=\frac 13.\]
Since $V_3$ is the quantization error for three-means, we have $V_3\leq \frac 13$. Let $\ga:=\set{a_1, a_2, a_3}$ be an optimal set of three-means. Since the points in an optimal set are the conditional expectations in their own Voronoi regions, we have $1\leq a_1<a_2<a_3<\infty$. Suppose that the Voronoi region of $a_1$ contains the set $\set{1, 2, 3}$.
Then,
\[V_3\geq \sum_{j=1}^3 f(j)(j-Av[1,3])^2=Er[1,3] =\frac{5}{14}>\frac 13\geq V_3,\]
which leads to a contradiction. Hence, we can assume that the Voronoi region of $a_1$ contains only the set $\set{1}$, or the set $\set{1,2}$. For the sake of contradiction, assume that the Voronoi region of $a_1$ contains only the set $\set{1}$.
Then, the Voronoi region of $a_2$ must contain the element $2$. Suppose that the Voronoi region of $a_2$ contains the set $\set{2, 3, 4, 5}$. Then,
\[V_3\geq  Er[2,5] =\frac{181}{368}=0.491848>V_3,\]
which yields a contradiction. Assume that the Voronoi region of $a_2$ contains only the set $\set{2,3,4}$, and so the Voronoi region of $a_3$ contains the set $\set{n : n\geq 5}$. Then, the distortion error is
\[V_3=Er[2,4]+Er[5, \infty)=\frac{9}{22} =0.409091>V_3,\]
which gives a contradiction.
Next, assume that the Voronoi region of $a_2$ contains only the set $\set{2,3}$, and so the Voronoi region of $a_3$ contains the set $\set{n : n\geq 4}$. Then, the distortion error is
\[V_3=Er[2,3]+Er[4, \infty)=\frac{7}{20}>V_3,\]
which leads to a contradiction. Finally, assume that the Voronoi region of $a_2$ contains only the set $\set{2}$, and so the Voronoi region of $a_3$ contains the set $\set{n : n\geq 3}$. Then, the distortion error is
\[V_3=Er[3, \infty)=\frac{1}{2}>V_3,\]
which gives a contradiction.
Thus, we can conclude that the Voronoi region of $a_1$ contains only the set $\set{1,2}$.
Then, the Voronoi region of $a_2$ must contain the element $3$. Suppose that the Voronoi region of $a_2$ contains the set $\set{3, 4, 5, 6}$. Then,
\[V_3\geq  \sum_{j=1}^2f(j)(j-Av[1,2])^2+ \sum_{j=3}^6f(j)(j-Av[3, 6])^2=Er[1,2]+Er[3,6] =\frac{59}{160} =0.36875>V_3,\]
which yields a contradiction. Assume that the Voronoi region of $a_2$ contains only the set $\set{3, 4, 5}$, and so the Voronoi region of $a_3$ contains the set $\set{n : n\geq 6}$. Then, the distortion error is
\[V_3=Er[1,2]+Er[3,5]+Er[6, \infty)=\frac{29}{84} =0.345238>V_3,\]
which gives a contradiction.
Next, assume that the Voronoi region of $a_2$ contains only the element $3$, and so the Voronoi region of $a_3$ contains the set $\set{n : n\geq 4}$. Then, the distortion error is
\[V_3=Er[1,2]+Er[4, \infty)=\frac{5}{12}=0.416667>V_3,\]
which yields a contradiction.
Hence, we can conclude that the Voronoi region of $a_2$ contains only the set $\set{3,4}$ yielding $a_1=Av[1,2]$, $a_2=Av[3,4]$, and $a_3=Av[5, \infty)$ with quantization error $V_3=\frac 13$.
Thus, the proof of the proposition is complete.
\end{proof}

\begin{prop1} \label{prop1121}
The sets $\set{1, 2,  Av[3,4], Av[5, \infty)}$  forms the optimal sets of four-means with quantization error $V_4=\frac{1}{6}$.
\end{prop1}
\begin{proof}
The distortion error due to set $\gb:=\set{1, 2, Av[3,4], Av[5, \infty)}$ is given by
\[V(P; \gb)=Er[3,4]+Er[5, \infty)=\frac 16.\]
Since $V_4$ is the quantization error for four-means, we have $V_4\leq \frac 16$. Let $\ga:=\set{a_1, a_2, a_3, a_4}$ be an optimal set of four-means. Since the points in an optimal set are the conditional expectations in their own Voronoi regions, we have $1\leq a_1<a_2<a_3<a_4<\infty$.
Clearly, the Voronoi region of $a_1$ contains the point $1$. Suppose that the Voronoi region of $a_1$ contains the set $\set{1, 2, 3}$.
Then,
\[V_3\geq \sum_{j=1}^3 f(j)(j-Av[1,3])^2=Er[1,3] =\frac{5}{14}>\frac 16\geq V_4,\]
which leads to a contradiction. Hence, we can assume that the Voronoi region of $a_1$ contains only the set $\set{1}$, or the set $\set{1,2}$. Suppose that the Voronoi region of $a_1$ contains only the set $\set{1,2}$. Then, the remaining elements of the set of natural numbers are contained in the union of the Voronoi regions of $a_2, a_3$ and $a_4$. Notice that the total distortion error contributed by the points $a_2, a_3$ and $a_4$ are positive. Hence,
\[V_4>\te{distortion error contributed by the point } a_1=Er[1, 2]=\frac 16=V_4,\]
which leads to a contradiction. Hence, the Voronoi region of $a_1$ cannot contain $\set{1, 2}$, i.e., the Voronoi region of $a_1$ contains only set $\set{1}$, i.e., $a_1=1$. Then, the Voronoi region of $a_2$ must contain $2$. Suppose that the Voronoi region of $a_2$ contains the set $\set{2,3,4}$. Then,
\[V_4\geq Er[2,4]=\frac{25}{88}>V_4,\]
which leads to a contradiction. Hence, we can assume that the Voronoi region of $a_2$ contains only the set $\set{2}$, or the set $\set{2,3}$. Suppose that the Voronoi region of $a_2$ contains only the set $\set{2,3}$. Assume that the Voronoi region of $a_3$ contains the set $\set{4,5,6,7}$. Then,
\[V_4\geq Er[2,3]+Er[4,7]=\frac{193}{960}>V_4,\]
which leads to a contradiction. Hence, we can assume that the Voronoi region of $a_3$ contains only the set $\set{4}$, $\set{4,5}$, or $\set{4,5,6}$. Suppose that the Voronoi region of $a_3$ contains only the set $\set{4,5,6}$. Then, the Voronoi region of $a_4$ contains the set $\set{n : n\geq 7}$. Then,
\[V_4=Er[2,3]+Er[4,6]+Er[7, \infty)=\frac{53}{280}>V_4,\]
which is a contradiction. Suppose that the Voronoi region of $a_3$ contains only the set $\set{4,5}$. Then, the Voronoi region of $a_4$ contains the set $\set{n : n\geq 6}$. Then,
\[V_4=Er[2,3]+Er[4,5]+Er[6, \infty)=\frac{11}{60}>V_4,\]
which leads to a contradiction.  Suppose that the Voronoi region of $a_3$ contains only the set $\set{4}$. Then, the Voronoi region of $a_4$ contains the set $\set{n : n\geq 5}$. Then,
\[V_4=Er[2,3]+Er[5, \infty)=\frac{9}{40}>V_4,\]
which leads to a contradiction.
Thus, we see that if the Voronoi region of $a_2$ contains only the set $\set{2,3}$, then a contradiction arises. Hence, we can conclude that the Voronoi region of $a_2$ contains only the set $\set{2}$, in other words, we have $a_2=2$. Then, the Voronoi region of $a_3$ contains the set $\set{3}$. Suppose that the Voronoi region of $a_3$ contains the set $\set{3,4,5, 6}$, then as before we see a contradiction arises. Hence, the Voronoi region of $a_3$ contains only the set $\set{3}$, $\set{3, 4}$, or the set $\set{3, 4,5}$. Notice that if the Voronoi region of $a_3$ contains only the set $\set{3, 4, 5}$, then the Voronoi region of $a_4$ contains the set $\set{n: n\geq 6}$, and if the Voronoi region of $a_3$ contains only the set $\set{3}$, then the Voronoi region of $a_4$ contains the set $\set{n: n\geq 4}$. In either of the cases, proceeding as before, we see that a contradiction arises. Hence, we can conclude that the Voronoi region of $a_3$ contains only the set $\set{3, 4}$. Hence, the Voronoi region of $a_4$ contains $\set{n : n\geq 5}$. Thus, we have
\[a_1=1, a_2=2, a_3=[3, 4], \te{ and } a_4=[5, \infty) \te{ with } V_4=\frac 16.\]
Thus, the proof of the proposition is complete.
\end{proof}

 \begin{prop1}
The sets  $\set{1, 2, \cdots, n-3, Av[n-2, n-1], Av[n, n+1], Av[n+2, \infty)}$ and $\set{1, 2, \cdots, n-3, n-2, Av[n-1, n], Av[n+1, \infty)}$ form the optimal sets of $n$-means for all $n\geq 5$ with the quantization error
$V_n=\frac{2^{3-n}}{3}.$
\end{prop1}
 \begin{proof}
The proof follows by Theorem~\ref{main} and Conjecture~\ref{conj1} under the assumption that Conjecture~\ref{conj1} is true.
 \end{proof}

\subsection{Optimal quantization for $P$ when $(p_1, p_2, p_3, p_4, \cdots)=(\frac 1 {2^3}, \frac 1 {2^2},  \frac 1{2}, \frac 1{2^4}, \frac 1{2^5}, \cdots)$}
Let us give the results in the following propositions.

\begin{prop1} \label{prop243}
The optimal set of two-means is given by $\set{Av[1,3], Av[4, \infty)}$ with quantization error $V_2=\frac{5}{7}$.
\end{prop1}

\begin{proof}
 We see that $Av[1,3]=\frac{17}{7}$, and $Av[4, \infty)=5$.
Since $3<\frac 12(\frac{17}{7}+5)=\frac{26}{7}<4$, the distortion error due to the set $\gb:=\set{\frac{17}{7}, 5}$ is given by
\[V(P; \gb)=Er[1,3]+Er[4, \infty)=\frac{5}{7}.\]
Since $V_2$ is the quantization error for two-means, we have $V_2\leq \frac{5}{7}$. Let $\ga:=\set{a_1, a_2}$ be an optimal set of two-means such that $a_1<a_2$. Since the points in an optimal set are the conditional expectations in their own Voronoi regions, we have $1\leq a_1<a_2<\infty$. Notice that the Voronoi region of $a_1$ must contain $1$. Suppose that the Voronoi region of $a_1$ contains the set $\set{1, 2, 3, 4,5}$. Then,
\[V_2\geq Er[1,5]=\frac{393}{496}>V_2,\]
which yields a contradiction. Thus, we can conclude that the Voronoi region of $a_1$ does not contain the point $5$.
Suppose that the Voronoi region of $a_1$ contains only the $\set{1, 2, 3, 4}$, and so the Voronoi region of $a_2$ contains the set $\set{n : n\geq 5}$. Then, we have
\[V_2=Er[1,4]+Er[5, \infty)=\frac{11}{15}>V_2,\]
which leads to a contradiction. Similarly, we can show that if the Voronoi region of $a_1$ contains only the set $\set{1}$, or the set $\set{1, 2}$, then we get a contradiction. Hence, we can assume that the Voronoi region of $a_1$ contains only the set $\set{1, 2, 3}$, and so the Voronoi region of $a_2$ contains only the set $\set{n : n\geq 4}$. Thus, we have
\[a_1=Av[1,3], \ a_2=Av[4, \infty) \te{ with } V_2=\frac{5}{7}.\]
Thus, the proof of the proposition is complete.
\end{proof}
\begin{prop1} \label{prop244}
The sets $\set{Av[1,2], Av[3,4], Av[5, \infty)}$  forms the optimal sets of three-means with quantization error $V_3=\frac{19}{72}$.
\end{prop1}

\begin{proof}
The distortion error due to set $\gb:=\set{Av[1,2], Av[3,4], Av[5, \infty)}$ is given by
\[V(P; \gb)=Er[1,2]+Er[3, 4]+Er[5, \infty)=\frac{19}{72}.\]
Since $V_3$ is the quantization error for three-means, we have $V_3\leq \frac{19}{72}=0.263889$. Let $\ga:=\set{a_1, a_2, a_3}$ be an optimal set of three-means. Since the points in an optimal set are the conditional expectations in their own Voronoi regions, we have $1\leq a_1<a_2<a_3<\infty$. Suppose that the Voronoi region of $a_1$ contains the set $\set{1, 2, 3}$.
Then,
\[V_3\geq Er[1,3] =\frac{13}{28}>\frac{19}{72}\geq V_3,\]
which leads to a contradiction. Hence, we can assume that the Voronoi region of $a_1$ contains only the set $\set{1}$, or the set $\set{1,2}$. Suppose that the Voronoi region of $a_1$ contains only the set $\set{1}$.
In this case, the Voronoi region of $a_2$ must contain the element $2$. Suppose that the Voronoi region of $a_2$ contains the set $\set{2, 3, 4}$.
Then,
\[V_3\geq  Er[2,4] =\frac{7}{26}=0.269231>V_3,\]
which yields a contradiction. Assume that the Voronoi region of $a_2$ contains only the set $\set{2,3}$, and so the Voronoi region of $a_3$ contains the set $\set{n : n\geq 4}$. Then, the distortion error is
\[V_3=Er[2,3]+Er[4, \infty)=\frac{5}{12}>V_3,\]
which leads to a contradiction. Finally, assume that the Voronoi region of $a_2$ contains only the set $\set{2}$, and so the Voronoi region of $a_3$ contains the set $\set{n : n\geq 3}$. Then, the distortion error is
\[V_3=Er[3, \infty)=\frac{13}{20}>V_3,\]
which leads to a contradiction.
Hence, we can assume that the Voronoi region of $a_1$ contains only the set $\set{1,2}$. Then, the Voronoi region of $a_2$ must contain the set $\set{3}$.
Suppose that the Voronoi region of $a_2$ contains the set $\set{3, 4, 5, 6}$. Then,
\[V_3\geq  Er[1,2]+Er[3,6] =\frac{151}{416}>V_3,\]
which yields a contradiction. Assume that the Voronoi region of $a_2$ contains only the set $\set{3, 4, 5}$, and so the Voronoi region of $a_3$ contains the set $\set{n : n\geq 6}$. Then, the distortion error is
\[V_3=Er[1,2]+Er[3,5]+Er[6, \infty)=\frac{35}{114}>V_3,\]
which gives a contradiction.
Next, assume that the Voronoi region of $a_2$ contains only the element $3$, and so the Voronoi region of $a_3$ contains the set $\set{n : n\geq 4}$. Then, the distortion error is
\[V_3=Er[1,2]+Er[4, \infty)=\frac 13>V_3,\]
which yields a contradiction.
Hence, we can conclude that the Voronoi region of $a_2$ contains only the set $\set{3,4}$ yielding $a_1=Av[1,2]$, $a_2=Av[3,4]$, and $a_3=Av[5, \infty)$ with quantization error $V_3=\frac{19}{72}$.
\end{proof}

\begin{prop1} \label{prop245}
The sets $\set{Av[1,2], 3, Av[4,5], Av[6, \infty)}$  forms the optimal sets of four-means with quantization error $V_4=\frac{1}{6}$.
\end{prop1}

\begin{proof}
The distortion error due to set $\gb:=\set{Av[1,2], 3, Av[4,5], Av[6, \infty)}$ is given by
\[V(P; \gb)=Er[1,2]+Er[4, 5]+Er[6, \infty)=\frac 16.\]
Since $V_4$ is the quantization error for four-means, we have $V_4\leq \frac 16=0.166667$.
Let $\ga:=\set{a_1, a_2, a_3, a_4}$ be an optimal set of four-means. Since the points in an optimal set are the conditional expectations in their own Voronoi regions, we have $1\leq a_1<a_2<a_3<a_4<\infty$. Suppose that the Voronoi region of $a_1$ contains the set $\set{1, 2, 3}$. Then,
\[V_4\geq Er[1,3] =\frac{13}{28}> V_4,\]
which leads to a contradiction. Hence, we can assume that the Voronoi region of $a_1$ contains only the set $\set{1}$, or the set $\set{1,2}$. Suppose that the Voronoi region of $a_1$ contains only the set $\set{1}$.
In this case, the Voronoi region of $a_2$ must contain the element $2$. Suppose that the Voronoi region of $a_2$ contains the set $\set{2, 3, 4}$.
Then,
\[V_4\geq  Er[2,4] =\frac{7}{26}=0.269231>V_4,\]
which yields a contradiction. Assume that the Voronoi region of $a_2$ contains only the set $\set{2,3}$. Then, notice that the Voronoi regions of $a_3$ and $a_4$ contain all the elements $\set{n : n\geq 4}$. Thus, the total distortion error contributed by $a_3$ and $a_4$ must be positive. This leads to the fact that
\[V_4>Er[2,3]=\frac 16\geq V_4,\]
which gives a contradiction. Assume that the Voronoi region of $a_2$ contains only the set $\set{2}$. Then, as before, we see that a contradiction arises.
Hence, we can assume that the Voronoi region of $a_1$ contains only the set $\set{1, 2}$. Then, the Voronoi region of $a_2$ must contain $3$. If the Voronoi region of $a_2$ contains more points using the similar arguments as before, we can show that a contradiction arises. Hence, we can conclude that $a_2=3$. Again, using the similar arguments, we can show that the Voronoi region of $a_3$ contains only the set $\set{4,5}$, and the Voronoi region of $a_4$ contains only the set $\set{n : n\geq 6}$. Thus, we have
\[a_1=Av[1,2], a_2=3, a_3=Av[4,5], \te{ and } a_4=Av[6, \infty) \te{ with quantization error } V_4=\frac 16.\]
Thus, the proof of the proposition is complete.
\end{proof}

\begin{prop1} \label{prop246}
The sets $\set{1, 2, 3, Av[4,5], Av[6, \infty)}$  forms the optimal sets of five-means with quantization error $V_5=\frac{1}{12}$.
\end{prop1}
\begin{proof}
The distortion error due to set $\gb:=\set{1, 2, 3, Av[4,5], Av[6, \infty)}$ is given by
\[V(P; \gb)=Er[4, 5]+Er[6, \infty)=\frac{1}{12}.\]
Since $V_5$ is the quantization error for five-means, we have $V_5\leq \frac{1}{12}=0.0833333$.
Let $\ga:=\set{a_1, a_2, a_3, a_4, a_5}$ be an optimal set of five-means such that $a_1<a_2<a_3<a_4<a_5$. Since the points in an optimal set are the conditional expectations in their own Voronoi regions, we have $1\leq a_1<a_2<a_3<a_4<a_5<\infty$. Clearly, the Voronoi region of $a_1$ contains the point $1$. For the sake of contradiction, assume that the Voronoi region of $a_1$ contains the set $\set{1,2,3}$. Then,
\[V_5\geq Er[1,3]=\frac{13}{28}>V_5,\]
which is a contradiction. Next, assume that the Voronoi region of $a_1$ contains only the set $\set{1, 2}$. Then, notice that the union Voronoi regions of $a_2$,  $a_3$, $a_4$, and $a_5$ contain all the elements $\set{n : n\geq 3}$. Hence, we must have
\[V_5>Er[1,2]=\frac{1}{12}\geq V_5,\]
which is a contradiction. Hence, we can conclude that the Voronoi region of $a_1$ contains only the element $1$, i.e., $a_1=1$. Clearly, the Voronoi region of $a_2$ contains the element $2$. Suppose the Voronoi region of $a_2$ contains the set $\set{2,3}$. Then, we have
\[V_5\geq Er[2,3]=\frac 16>V_5,\]
which give a contradiction. Hence, the Voronoi region of $a_2$ contains only the element $2$, i.e., $a_2=2$. Similarly, we can show that $a_3=3$. The rest of the proof follows in the similar lines as given in Proposition~\ref{prop1121}. Thus, we see that $a_4=Av[4,5]$ and $a_5=Av[6, \infty)$ with quantization error $V_5=\frac{1}{12}$. Thus, the proof of the proposition is complete.
\end{proof}

 \begin{prop1}
The sets  $\set{1, 2, \cdots, n-3, Av[n-2, n-1], Av[n, n+1], Av[n+2, \infty)}$ and $\set{1, 2, \cdots, n-3, n-2, Av[n-1, n], Av[n+1, \infty)}$ form the optimal sets of $n$-means for all $n\geq 6$ with the quantization error
$V_n=\frac{2^{3-n}}{3}.$
\end{prop1}
 \begin{proof}
The proof follows by Theorem~\ref{main} and Conjecture~\ref{conj1} under the assumption that Conjecture~\ref{conj1} is true.
 \end{proof}

\section{Observation and Remarks}
In Conjecture~\ref{conj1} the probability measure $P$ is defined as  $P:=\sum_{j=1}^{k-1} p_j\gd_j+\sum_{j=k}^\infty \frac 1{2^j}\gd_j$, where $(p_1, p_2, \cdots, p_{k-1})$ is a permutation of the set $\set{\frac 12, \frac 1{2^2}, \cdots, \frac 1{2^{k-1}}}$, where $k\in \D N$ with $k\geq 2$. If $P:=\sum_{j=1}^{k-1} p_j\gd_j+\sum_{j=k}^\infty \frac 1{2^j}\gd_j$, and $(p_1, p_2, \cdots, p_{k-1})$ is not a permutation of the set $\set{\frac 12, \frac 1{2^2}, \cdots, \frac 1{2^{k-1}}}$, then Conjecture~\ref{conj1} is not true. In this regard, we give the following proposition.

\begin{prop} \label{prop51}
For the probability measure $P$ given by $P:=\frac {149}{200} \gd_1+ \frac 1{200}\gd_2+\sum_{j=3}^\infty \frac 1{2^j} \gd_j$ the optimal set of five-means is given by
\[\set{1, Av[2,3], 4, Av[5,6], Av[7, \infty)}, \te{ or } \set{1, Av[2,3], Av[4,5], Av[6,7], Av[8, \infty)}\] with quantization error $V_5=\frac{29}{624}$.
\end{prop}
\begin{proof}
The distortion error due to set $\gb:=\set{1, Av[2,3], 4, Av[5,6], Av[7, \infty)}$ is given by
\[V(P; \gb)=Er[2,3]+Er[5,6]+Er[7, \infty)=\frac{29}{624}.\]
Since $V_5$ is the quantization error for five-means, we have $V_5\leq \frac{29}{624}=0.0464744$.
Let us assume that  $\ga:=\set{a_1, a_2, a_3, a_4, a_5}$ is an optimal set of five-means such that $a_1<a_2<a_3<a_4<a_5$. Since the points in an optimal set are the conditional expectations in their own Voronoi regions, we have $1\leq a_1<a_2<a_3<a_4<a_5<\infty$. Clearly, the Voronoi region of $a_1$ contains the point $1$. For the sake of contradiction, assume that the Voronoi region of $a_1$ contains the set $\set{1,2,3}$. Then,
\[V_5\geq Er[1,3]=\frac{7537}{17500}=0.430686>V_5,\]
which is a contradiction. Hence, we can assume that the Voronoi region of $a_1$ contains only the set $\set{1}$ or the set $\set{1,2}$. Suppose that the Voronoi region of $a_1$ contains only the set $\set{1,2}$. Then, the Voronoi region of $a_2$ must contain the element $3$. Suppose that the Voronoi region of $a_2$ contains the set $\set{3,4}$. Then,
\[V_5\geq Er[1,2]+Er[3,4]=\frac{1399}{30000}=0.0466333>V_5,\]
which leads to a contradiction. Hence, we can assume that the Voronoi region of $a_2$ contain only the element $3$, i.e., $a_2=3$. Then, the union of the Voronoi regions of $a_3, a_4, a_5$ contains the set $\set{4,5,6, \cdots}$ with associated probability $\frac 1{2^j}$ for each $j\in \set{4,5,6, \cdots}$. Hence, using the similar lines as described in the proof of Theorem~\ref{main}, we can show that
\begin{align} \label{exp1} \set{a_3, a_4, a_5} \te{ equals the set } \set{4, Av[5,6], Av[7, \infty)}, \te{ or } \set{Av[4,5], Av[6,7], Av[8, \infty)}
\end{align} with the quantization error
\[V_5=Er[1,2]+Er[5,6]+Er[7, \infty)=\frac{1399}{30000}=0.0466333>V_5,\]
which leads to a contradiction. Hence, we can assume that the Voronoi region of $a_1$ contains only the element $1$, i.e., $a_1=1$. Then, the Voronoi region of $a_2$ must contain $2$. Suppose that the Voronoi region of $a_2$ contains the set $\set{2,3,4}$. Then,
\[V_5\geq Er[2,4]=\frac{31}{616}>V_5,\]
which leads to a contradiction. Hence, the Voronoi region of $a_2$ contains only the set $\set{2}$, or $\set{2,3}$. Suppose that the Voronoi region of $a_2$ contains only the set $\set{2}$, i.e., $a_2=2$.
Then, as $a_1=1$, $a_2=2$, using the similar lines as described in the proof of Theorem~\ref{main}, we can show that $a_3=3, a_4=Av[4,5], a_5=Av[6, \infty)$; or $a_3=Av[3,4], a_4=Av[5,6], a_5=Av[7, \infty)$ with quantization error $V_5=\frac{1}{12}>V_5$, which is a contradiction. Hence, we can assume that the Voronoi region of $a_2$ contains only the set $\set{2,3}$.  Again, using the similar lines as described in the proof of Theorem~\ref{main}, we can show that $\set{a_3, a_4, a_5}$ equals the set $\set{4, Av[5,6], Av[7, \infty)}$; or $\set{Av[4,5], Av[6,7], Av[8, \infty)}$. Thus, we conclude that the optimal set of five-means is
either $\set{1, Av[2,3], 4, Av[5,6], Av[7, \infty)}$ or $\set{1, Av[2,3], Av[4,5], Av[6,7], Av[8, \infty)}$ with quantization error $V_5=\frac{29}{624}$.
This completes the proof of the proposition.
\end{proof}

\begin{remark}
Proposition~\ref{prop51} implies that Conjecture~\ref{conj1} is not true for an arbitrary probability distribution $(p_1, p_2, p_3, \cdots)$ associated with the set of positive integers $\D N$.
\end{remark}

\begin{remark}
 Conjecture~\ref{conj1} is verified by two examples given in Subsection~\ref{subsection111} and Subsection~\ref{subsection112}. We still could not give a general proof of the conjecture. It will be worthwhile to investigate the general proof of the conjecture.
 \end{remark}

\end{document}